\documentclass[12pt]{amsproc}

\usepackage{amsmath, amsthm, amssymb, amsfonts, amscd, graphicx, endnotes, tikz, color, hyperref}

\usepackage{mathtools}

\usepackage[margin=01.3in]{geometry}

\usetikzlibrary{arrows}

\def\CC{{\mathbb C}}

\def\QQ{{\mathbb Q}}
\def\PP{{\mathbb P}}
\def\QQ{{\mathbb Q}}
\def\RR{{\mathbb R}}

\def\ZZ{{\mathbb Z}}

\def\hhat{{\hat h}}

\def\0{{\mathbf 0}}
\def\1{{\mathbf 1}}

\def\Abf{{\mathbf A}}

\def\Pbf{{\mathbf P}}

\def\Ocal{{\mathcal O}}

\def\Kbar{{\bar K}}

\def\Qbar{\overline{\QQ}}

\def\Gal{\mathrm{Gal}}

\def\sup{\mathrm{sup}}
\def\max{\mathrm{max}}

\theoremstyle{plain}

\newtheorem{thm}{Theorem}

\newtheorem{lem}[thm]{Lemma}

\theoremstyle{definition}

\newtheorem*{rem}{Remark}
\newtheorem{ex}{Example}

\title[Small totally $p$-adic algebraic numbers]{A dynamical construction of small \\ totally $p$-adic algebraic numbers}

\author{Clayton Petsche}

\address{Clayton Petsche; Department of Mathematics; Oregon State University; Corvallis OR 97331 U.S.A.}

\email{petschec@math.oregonstate.edu}

\author{Emerald Stacy}

\address{Emerald Stacy; Department of Mathematics and Computer Science; Washington College; 300 Washington Avenue; Chestertown MD 21620 U.S.A.}

\email{estacy2@washcoll.edu}

\date{January 5 2019}

\begin{document}

\begin{abstract}
We give a dynamical construction of an infinite sequence of distinct totally $p$-adic algebraic numbers whose Weil heights tend to the limit $\frac{\log p}{p-1}$, thus giving a new proof of a result of Bombieri-Zannier.  The proof is essentially equivalent to the explicit calculation of the Arakelov-Zhang pairing of the maps $\sigma(x)=x^2$ and $\phi_p(x)=\frac{1}{p}(x^p-x)$.
\end{abstract}

\maketitle


\section{Introduction}

Let $h:\Qbar\to\RR$ denote the absolute Weil height function, and given a subfield $L$ of $\Qbar$, define
\begin{equation}\label{LimInfDef}
\liminf_{\alpha\in L}h(\alpha)
\end{equation}
to be the unique extended real number $B_0\in[0,+\infty]$ with the property that the set
$$
\{\alpha\in L \mid h(\alpha)\leq B\}
$$
is finite for all $B<B_0$ and infinite for all $B>B_0$.  For example, when $L/\QQ$ is a finite extension we have $\liminf_{\alpha\in L}h(\alpha)=+\infty$ by the Northcott finiteness principle.  The converse is not true, as there do exist algebraic extensions $L/\QQ$ of infinite degree such that the set $\{\alpha\in L\mid h(\alpha)\leq B\}$ is finite for all $B>0$; for example, the compositum of all quadratic extensions of $\QQ$ is one such field (Bombieri-Zannier \cite{MR1898444}).  

Let $p$ be a prime.  Denote by $\QQ_{(p)}$ the subfield of $\Qbar$ consisting of all totally $p$-adic numbers; that is, those algebraic numbers whose minimal polynomials over $\QQ$ split completely over the field $\QQ_p$ of $p$-adic numbers.  Bombieri-Zannier \cite{MR1898444} proved that
\begin{equation}\label{BombZannBound}
\frac{\log p}{2(p+1)}\leq \liminf_{\alpha\in \QQ_{(p)}}h(\alpha) \leq\frac{\log p}{p-1}.
\end{equation}
In particular, $\QQ_{(p)}$ has what Bombieri-Zannier call the Bogomolov property, which means that $h(\alpha)\geq C>0$ for some constant $C$ and all nonzero, non-root-of-unity $\alpha\in \QQ_{(p)}$.  The lower bound in  (\ref{BombZannBound}) was later improved slightly by Fili-Petsche \cite{MR3340356}, who showed that 
$\liminf_{\alpha\in \QQ_{(p)}}h(\alpha) \geq \frac{p\log p}{2(p^2-1)}$, and more significantly by Pottmeyer \cite{MR3869606}, who showed that $\liminf_{\alpha\in \QQ_{(p)}}h(\alpha) \geq \frac{\log(p/2)}{p+1}$ and $\liminf_{\alpha\in \QQ_{(2)}}h(\alpha) \geq \frac{\log2}{4}$.

To prove the upper bound in (\ref{BombZannBound}), Bombieri-Zannier used a fairly intricate construction which begins with the polynomial $(x-1)(x-2)\dots(x-r)$ for arbitrarily large $r\geq1$, then performing a perturbation, which on the one hand is $p$-adically small enough to preserve complete splitting over $\QQ_p$, but which on the other hand obtains irreducibility and satisfies the needed asymptotic height bound.  Fili \cite{MR3232769} has given a new proof of the upper bound in (\ref{BombZannBound}) whose main ingredient is the powerful Fekete-Szego theorem with splitting conditions due to Rumely \cite{MR1904930}.

In this paper we give a new proof of the upper bound in (\ref{BombZannBound}) using totally different methods.  Our approach uses ideas from arithmetic dynamical systems and is inspired by the method used by Smyth \cite{MR607924} in the totally real setting.  We divide our results into two theorem statements.

\begin{thm}\label{MainThm1}
Let $p$ be a prime and define the polynomial $\phi_p(x)\in\QQ[x]$ by $\phi_p(x)=\frac{1}{p}(x^p-x)$.   For each $n\geq1$ let $\alpha_n\in\Qbar$ satisfy $\phi_p^n(\alpha_n)=1$.  Then $\{\alpha_n\}$ is a sequence of distinct totally $p$-adic algebraic numbers such that $h(\alpha_n) \leq \frac{\log (p+1)}{p-1}$ for all $n\geq1$.  In particular,
\begin{equation}\label{MainThm1Bound}
\liminf_{\alpha\in \QQ_{(p)}}h(\alpha) \leq \frac{\log (p+1)}{p-1}.
\end{equation}
\end{thm}

\begin{thm}\label{MainThm2}
The sequence $\{\alpha_n\}$ of totally $p$-adic algebraic numbers constructed in Theorem~\ref{MainThm1} satisfies $h(\alpha_n)\to\frac{\log p}{p-1}$.  In particular,
\begin{equation}\label{MainThm2Bound}
\liminf_{\alpha\in \QQ_{(p)}}h(\alpha) \leq \frac{\log p}{p-1}.
\end{equation}
\end{thm}

The bound (\ref{MainThm2Bound}) is stronger than (\ref{MainThm1Bound}) and thus supercedes it, but (\ref{MainThm1Bound}) is only very slightly worse, and Theorem~\ref{MainThm1} has the advantage of a fairly simple and self-contained proof, using only Hensel's Lemma and elementary dynamical arguments. 

The proof of Theorem~\ref{MainThm2} is non-elementary and requires the machinery surrounding the Call-Silverman canonical height function, analysis on the Berkovich projective line, and the equidistribution theorem for dynamically small points due to Baker-Rumely \cite{MR2244226}, Chambert-Loir \cite{MR2244803}, and Favre-Rivera-Letelier \cite{MR2221116}.  While it does not improve upon the results of Bombieri-Zannier and Fili, it is intriguing that the three completely different methods lead to precisely the same bound.

An aspect of Theorem~\ref{MainThm2} which may be interesting to arithmetic dynamicists is that its proof gives explicit calculation of a nonzero Arakelov-Zhang pairing.  Given two rational maps $\phi(x),\psi(x)\in K(x)$ of degree at least two defined over a number field, the Arakelov-Zhang pairing $\langle\phi,\psi\rangle$ is a nonnegative real number which measures the arithmetic-dynamical distance between the two maps; this pairing has been used recently by researchers working in the area of unlikely intersections; see e.g. \cite{Fili2017}, \cite{DKY2019A}, \cite{DKY2019B}.  Petsche-Szpiro-Tucker \cite{MR2869188} have shown that if $\{\alpha_n\}$ is a sequence of distinct points in $\PP^1(\Kbar)$ with $\hhat_{\phi}(\alpha_n)\to0$ (where $\hhat_\phi$ is the Call-Silverman canonical height function), then $\hhat_{\psi}(\alpha_n)\to \langle\phi,\psi\rangle$.  Thus Theorem~\ref{MainThm2} is equivalent to the calculation
\begin{equation}\label{AZCalc}
\langle\sigma,\phi_p\rangle=\frac{\log p}{p-1}
\end{equation}
where $\sigma(x)=x^2$ and  $\phi_p(x)=\frac{1}{p}(x^p-x)$.  We are not aware of any other examples in the literature in which a nonzero dynamical Arakelov-Zhang pairing is explicitly calculated in closed form.

In $\S$~\ref{MainThm1Sect} we review the definition the Weil height and we prove Theorem~\ref{MainThm1}.  In $\S$~\ref{ReviewLocalSect} we give a review of polynomial dynamics on the Berkovich projective line, and in  $\S$~\ref{MainThm2Sect} we prove Theorem~\ref{MainThm2}.  The exposition of $\S$~\ref{ReviewLocalSect} and the remark at the end of $\S$~\ref{MainThm2Sect} may be of more general interest as they give a description of the Arakelov-Zhang pairing for polynomials which is somewhat more accessible than more general treatments currently in the literature.

\section{The proof of Theorem~\ref{MainThm1}}\label{MainThm1Sect}

Denote by $M_\QQ=\{\infty,2,3,5,\dots\}$ the set of places of $\QQ$.  For each $v\in M_\QQ$, let $\CC_v$ be the completion of the algebraic closure of $\QQ_v$, and let $|\cdot|_v$ be the absolute value on $\CC_v$ which is normalized so that it coincides with either the standard real or $p$-adic absolute value when restricted to $\QQ$.  

The absolute Weil height $h:\Qbar\to\RR$ may be defined as follows.  Given $\alpha\in\Qbar$, denote by $\Ocal_\alpha$ the $\Gal(\Qbar/\QQ)$-orbit of $\alpha$ in $\Qbar$.  For each place $v\in M_\QQ$, we may view $\Ocal_\alpha$ as a subset of $\CC_v$ via any embedding $\Qbar\hookrightarrow\CC_v$, and 
\begin{equation}\label{WeilHeightDef}
h(\alpha) = \frac{1}{[\QQ(\alpha):\QQ]}\sum_{v\in M_\QQ}\sum_{\beta\in\Ocal_\alpha}\log^+|\beta|_v,
\end{equation}
where $\log^+t=\log\max(1,t)$. 

\begin{lem}\label{phiPto1Lem}
The polynomial $\phi_p(x) = \frac{1}{p}(x^p-x)$ defines a surjective $p$-to-$1$ map $\phi_p : \ZZ_p \to \ZZ_p$.
\end{lem}

\begin{proof}
Given $\alpha \in \ZZ_p$, by Fermat's Little Theorem we have $|\alpha^p-\alpha|_p \leq 1/p$, and therefore $|\phi_p(\alpha)|_p \leq1$, establishing that $\phi_p(\ZZ_p) \subseteq \ZZ_p$.  Given $\beta \in \ZZ_p$, we want to show that there are $p$ distinct points $\alpha\in\ZZ_p$ that satisfy $\phi_p(\alpha) = \beta$.  For this we consider the polynomial 
$$
g(x) = x^p-x-p\beta.
$$
Let $a \in \{0,1,2,\cdots, p-1\}$.  By Fermat's Little Theorem and the strong triangle inequality,
\begin{equation*}
\begin{split}
| g(a) |_p & = |(a^p-a)-p\beta|_p \leq \tfrac{1}{p} \\
|g'(a)|_p & =| pa^{p-1}-1|_p = 1.
\end{split}
\end{equation*} 
For each such $a$, Hensel's Lemma indicates that there is a unique root of $\phi_p(x)=\beta$ in $\ZZ_p$ congruent to $a$ modulo $p$.  As a polynomial of degree $p$, there can be at most $p$ such roots, and thus $\phi_p:\ZZ_p\to\ZZ_p$ is surjecive and $p$-to-$1$.
\end{proof}

\begin{proof}[Proof of Theorem~\ref{MainThm1}]
To see that the $\alpha_n$ are distinct, suppose on the contrary that there exists some $\alpha\in\Qbar$ such that $\phi_p^m(\alpha) = \phi_p^{n}(\alpha) =1$ with $m<n$.
Then 
$$
1 = \phi_p^{n}(\alpha) = \phi_p^{n-m}(\phi_p^m(\alpha))=\phi_p^{n-m}(1),
$$
implying that $1$ is a periodic point with respect to $\phi_p$.  However, $\phi_p(1) = 0$ and $\phi_p(0)=0$, so $1$ is strictly preperiodic, and hence not periodic, a contradiction.

Next we observe that for each $n\geq1$, all of the algebraic conjugates $\beta\in\Ocal_{\alpha_n}$ of $\alpha_n$ are elements of $\phi^{-k}(1)$, because $\alpha_n$ is a root of $\phi^n(x)-1$ and hence its minimal polynomial over $\QQ$ is a divisor of $\phi^n(x)-1$.  The polynomial $\phi^n(x)-1$ has $p^n$ distinct roots in $\ZZ_p$, as follows from Lemma~\ref{phiPto1Lem}, and we can conclude that $\alpha_n$ is totally $p$-adic.

It remains only to give the bound on the height of the $\alpha_n$.  We will show that for all places $v\in M_\QQ$ and all $\beta\in\Ocal_{\alpha_n}\subseteq\CC_v$, we have
\begin{equation}\label{AbsValBounds}
\begin{split}
|\beta|_v & \leq 1 \text{ if } v \neq \infty \\
|\beta|_v & \leq (p+1)^{1/(p-1)} \text{ if } v = \infty.
\end{split}
\end{equation}

Suppose first that $v\not\in\{p,\infty\}$.  The polynomial $f(x) = \phi_p^n(x)-1$ has $v$-integral coefficients, and its leading coefficient, a power of $1/p$, is a $v$-adic unit.  As $\beta$ is a root of $f(x)$, it is $v$-integral; i.e. $|\beta|_v \leq 1$.

In the case $v=p$, by Lemma \ref{phiPto1Lem}, all points in the backwards orbit of $1$ are in $\ZZ_p$, and since $\beta\in\Ocal_{\alpha_n}\subseteq\phi_p^{-n}(1)$, we conclude that $|\beta|_p \leq 1$.

We turn to the case $v=\infty$.  Since $\phi_p^n(\beta)=1$, $\phi_p(1)=0$, and $\phi_p(0)=0$, $\beta$ is $\phi_p$-preperiodic, and so in particular the sequence $|\phi_p^k(\beta)|$ is bounded in $\CC$ as $k\to+\infty$.  Therefore, to prove (\ref{AbsValBounds}) it suffices to show that if $x \in \CC$ satisfies $|x| > (p+1)^{1/(p-1)}$, then $|\phi_p^k(x)| \to + \infty$ as $k \to +\infty$.  For such an $x$  we may select $\epsilon > 0$ so small that $|x| > (p(1+\epsilon)+1)^{1/(p-1)}$.  Then 
\begin{equation*}
\begin{split}
| \phi_p(x)  |_v  &= |\tfrac{1}{p}(x^p-x)  | \\
&= \tfrac{1}{p}|x| |x^{p-1}-1| \\
&\geq \tfrac{1}{p}|x|  ( |x|^{p-1}-1 )  \\
&\geq |x| (1+\epsilon).
\end{split}
\end{equation*}
Iterating this inequality gives $|\phi_p^k(x)| \geq |x| (1+\epsilon)^k$ and therefore  $|\phi_p^k(x)| \to + \infty$ as $k \to +\infty$.

Finally, using (\ref{AbsValBounds}) and the definition (\ref{WeilHeightDef}) of the height we have
$$ 
h(\alpha_n) = \frac{1}{[\QQ(\alpha_n):\QQ]}\sum_{v\in M_\QQ}\sum_{\beta\in\Ocal_{\alpha_n}}\log^+|\beta|_v \leq\frac{\log (p+1)}{p-1}.
$$
\end{proof}

\section{Review of local polynomial dynamics}\label{ReviewLocalSect}

For each place $v\in M_\QQ$, let $\Abf_v^1$ be the Berkovich affine line over $\CC_v$; thus $\Abf_v^1$ is the space (equipped with a natural topology) of multiplicative seminorms on the ring $\CC_v[T]$ extending the absolute value $|\cdot|_v$ on $\CC_v$.  If $v=\infty$, then $\Abf_v^1$ is naturally identitifed with $\CC_v=\CC$, but if $v$ is non-Archimedean, then $\Abf_v^1$ is a path-connected, locally compact space containing $\CC_v$ as a dense subspace.  Let $\Pbf_v^1=\Abf_v^1\cup\{\infty\}$ be the Berkovich projective line, the one-point compactification of $\Abf_v^1$.  If $v$ is non-Archimedean and $x\in\Abf_v^1$, then for any polynomial $f(T)\in\CC_v[T]$ we set
$$
|f(x)|_v=[f(T)]_x.
$$
In other words, we interpret $|f(x)|_v$ to mean the seminorm $[\cdot]_x$ of $\CC_v[T]$ evaluated at the polymonial $f(T)$.  For more background on $\Abf_v^1$ and $\Pbf_v^1$ see \cite{MR2599526} Ch. 2 or \cite{MR1070709}.

The measure-valued Laplacian $\Delta$ on $\Pbf_v^1$, in the sense of Baker-Rumely (\cite{MR2599526} Ch. 5 and \cite{MR2244226}), is a linear operator which assigns to each (suitably regular) function $f:\Pbf_v^1\to\RR\cup\{\pm\infty\}$ a signed, finite measure $\Delta f$ on $\Pbf_v^1$.  When $v=\infty$, thus when $\Pbf_v^1=\CC\cup\{\infty\}$, and for smooth $f:\CC\cup\{\infty\}\to\RR$, we have $\Delta f=-\frac{1}{2\pi}(\partial_{xx}+\partial_{yy})f(z)dx\,dy$.  In the non-Archimedean case, $\Delta$ is defined and its properties developed in \cite{MR2599526} Ch. 5.  We point out that the Laplacian is self-adjoint in the sense that $\int_{\Pbf_v^1} f d(\Delta g)=\int_{\Pbf_v^1} g d(\Delta f)$ whenever $f:\Pbf_v^1\to\RR\cup\{\pm\infty\}$ is $(\Delta g)$-integrable and $g:\Pbf_v^1\to\RR\cup\{\pm\infty\}$ is $(\Delta f)$-integrable; in the Archimedean case this is a standard application of multivariable integration by parts, while in the non-Archimedean case it is proved in \cite{MR2599526} Cor. 5.39.

Let $\phi(x)\in \CC_v[x]$ be a polynomial of degree $d\geq2$.  We now describe three closely related dynamical invariants associated to $\phi$, the first being the {\em filled-Julia set}
\begin{equation}\label{FilledJuliaDef}
F_{\phi,v} = \{x\in\Abf_v^1 \mid |\phi^n(x)|_v \text{ is bounded as } n\to+\infty\},
\end{equation}
and the second being the {\em Call-Silverman canonical local height function}
\begin{equation}\label{LocalHeightDef}
\hhat_{\phi,v}:\Pbf_v^1\to\RR\cup\{+\infty\} \hskip1cm \hhat_{\phi,v}(x) = \lim_{n\to+\infty}\frac{1}{d^n}\log^+|\phi^n(x)|_v
\end{equation}
(interpreting $\hhat_{\phi,v}(\infty) =+\infty$).  A standard telescoping series argument shows the existence of the limit in (\ref{LocalHeightDef}).  Finally, the  {\em canonical measure} $\mu_{\phi,v}$ associated to $\phi$ is a $\phi$-invariant unit Borel measure on $\Pbf_v^1$ which occurs as the limiting distribution in several important dynamical equidistribution results, one of which we will describe in the next section.  There are several equivalent constructions of this measure in the literature; see \cite{MR736568}, \cite{MR741393} in the Archimedean case and \cite{MR2244226}, \cite{MR2244803}, \cite{MR2221116} in the non-Archinedean case.  Here we follow the construction of Baker-Rumely \cite{MR2244226} $\S$~3, who characterize $\mu_{\phi,v}$ as the unit Borel measure on $\Pbf_v^1$ satisfying the identity
\begin{equation}\label{CanMeasDef}
\Delta\hhat_{\phi,v} = \delta_\infty-\mu_{\phi,v}
\end{equation}
where $\delta_\infty$ is the Dirac measure on $\Pbf_v^1$ supported at $\infty$.

The following lemma collects well-known properties of $F_{\phi,v}$, $\hhat_{\phi,v}$, and $\mu_{\phi,v}$ which will be useful in the proof of Theorem~\ref{MainThm2}.  When $v\in M_\QQ$ is non-Archimedean, let $\Ocal_v=\{a\in\CC_v\mid |a|_v\leq1\}$ be the ring of integers of $\CC_v$, and let $\zeta_0$ denote the {\em Gauss point} of $\Abf^1_v$; this is the point whose corresponding seminorm is the supremum norm on the unit disc of $\CC_v$; that is $[f(T)]_{\zeta_{0}}=\sup_{|x|_v\leq1}|f(T)|_v$. 

\begin{lem}\label{LocalHeightLem}
Let $\phi(x)\in \CC_v[x]$ be a polynomial of degree $d\geq2$, and let $x\in\Abf_v^1$. 
\begin{itemize}
\item[{\sc (a)}]  $\hhat_{\phi,v}(x)\geq0$, with $\hhat_{\phi,v}(x)=0$ if and only if $x\in F_{\phi,v}$.  
\item[{\sc (b)}] $\hhat_{\phi,v}(\phi(x))=d\hhat_{\phi,v}(x)$   
\item[{\sc (c)}] The support of the measure $\mu_{\phi,v}$ is contained in $F_{\phi,v}$.  
\item[{\sc (d)}]  If $\phi(x)=ax^d+\dots$, then $\lim_{x\to\infty}(\hhat_{\phi,v}(x)-\log|x|_v)=\frac{\log|a|_v}{d-1}$.
\item[{\sc (e)}]  If $v$ is non-Archimedean, and if $\phi(x)$ has coefficients in $\Ocal_v$ and leading coefficient in $\Ocal_v^\times$, then $F_{\phi,v}=\Ocal_v$, $\hhat_{\phi,v}(x)=\log^+|x|_v$, and $\mu_{\phi,v}=\delta_{\zeta_0}$ is the Dirac measure supported at the Gauss point of $\Abf_v^1$.
\end{itemize}
\end{lem}

These facts are all well-known: for parts {\sc (a)}, {\sc (b)}, and {\sc (e)} see \cite{MR2316407} $\S$~5.9 and \cite{MR2599526} Ch. 10.  Part {\sc (c)} follows from the fact that $\mu_{\phi,v}$ is supported on the Julia set of $\phi$, which is contained in $F_{\phi,v}$, see \cite{MR2599526} Ch. 10 in the non-Archimedean case and \cite{MR741393} in the Archimedean case.  Part {\sc (d)} follows from a standard telescoping series calculation.

\section{Global dynamics and the proof of Theorem~\ref{MainThm2}}\label{MainThm2Sect}

Suppose now that $\phi(x)\in \QQ[x]$ is a polynomial of degree $d\geq2$.  The Call-Silverman \cite{callsilv:htonvariety} canonical height function $\hhat_\phi:\Qbar\to\RR$ is given by the limit $\hhat_\phi(\alpha)=\lim_{n\to+\infty}h(\phi^n(\alpha))/d^n$, and may be characterized by the two properties: (i) $\hhat_\phi(\alpha)=h(\alpha)+O(1)$ for all $\alpha\in\Qbar$; and (ii) $\hhat_\phi(\phi(\alpha))=d\hhat_\phi(\alpha)$ for all $\alpha\in\Qbar$.  Locally, the canonical height may be expressed via local heights in a formula analogous to (\ref{WeilHeightDef}), as follows.  Given $\alpha\in\Qbar$, denote by $\Ocal_\alpha$ the $\Gal(\Qbar/\QQ)$-orbit of $\alpha$ in $\Qbar$, and for each place $v\in M_\QQ$, view $\Ocal_\alpha$ as a subset of $\CC_v$ via any embedding $\Qbar\hookrightarrow\CC_v$; then 
\begin{equation*}
\hhat_\phi(\alpha) = \frac{1}{[\QQ(\alpha):\QQ]}\sum_{v\in M_\QQ}\sum_{\beta\in\Ocal_\alpha}\hhat_{\phi,v}(\alpha).
\end{equation*}

The main relationship between the canonical height $\hhat_\phi$ and the family of local canonical measures $\mu_{\phi,v}$ is the equidistribution theorem for dynamically small points in $\Qbar$, proved by Baker-Rumely \cite{MR2244226}, Chambert-Loir \cite{MR2244803}, and Favre-Rivera-Letelier \cite{MR2221116}.  A special case of this result states that if $\{\alpha_n\}$ is a sequence in $\Qbar$ with $\hhat_\phi(\alpha_n)\to0$, then for each place $v\in M_\QQ$, the sequence of $\Gal(\Qbar/\QQ)$-orbits $\{\Ocal_{\alpha_n}\}$ is $\mu_{\phi,v}$-equidistributed.  This equidistriubtion theorem is the main ingredient in the proof of Theorem~\ref{MainThm2}.  

\begin{ex}\label{SquaringEx}
Consider the squaring map $\sigma(x)=x^2$.  Then $\hhat_{\sigma,v}(x)=\log^+|x|_v$ for all $v\in M_\QQ$, and thus $\hhat_\sigma(\alpha)=h(\alpha)$.  At the Archimedean place, $\mu_{\sigma,\infty}$ is the measure supported on the unit circle group of $\CC_\infty=\CC$ where it coincides with normalized Haar measure.  If $v$ is non-Archimedean then $\mu_{\phi,v}$ is the Dirac measure $\delta_{\zeta_{0}}$ supported at the Gauss point $\zeta_{0}$ of $\Abf^1_v$.  See \cite{MR2244226} Ex. 3.43 and Ex. 5.4.
\end{ex}

\begin{ex}\label{PhiPExample}
Consider the map $\phi_p(x)=\frac{1}{p}(x^p-x)$.  By Lemma~\ref{LocalHeightLem}, if $v\not\in\{p,\infty\}$, we have $\hhat_{\phi_p,v}(x)=\log^+|x|_v$.  At the place $v=p$, it was pointed out in \cite{MR2599526} Ex. 10.120 that $F_{\phi_p,p}=\ZZ_p$ and that $\mu_{\phi_p,p}$ is normalized Haar measure on $\ZZ_p$.  However, for the purposes of Theorem~\ref{MainThm2} we do not need the full strength of this result.  Instead, we only need to point out that 
$F_{\phi_p,p}\subseteq\{x\in\Abf_p^1\mid |x|_p\leq1\}$.  To see this, if $|x|_p>1$, then $|\phi_p(x)|_p=|\frac{1}{p}(x^p-x)|_p=p|x|_p^p$, and iterating, $|\phi_p^n(x)|_p=p^{1+p+\dots+p^{n-1}}|x|_p^{p^n}$.  In particular, $|\phi_p^n(x)|_p\to+\infty$ and therefore $x\not\in F_{\phi_p,p}$.  We point out that this calculation also shows that when $|x|_v>1$, we have $\hhat_{\phi_p,p}(x)=\frac{\log p}{p-1}+\log|x|_p$.  It would be interesting to give an explicit calculation of $\hhat_{\phi_p,p}(x)$ for $x\in\Abf_v^1$ with $|x|_v\leq 1$.
\end{ex}

\begin{proof}[The proof of Theorem~\ref{MainThm2}]  For each place $v\in M_\QQ$, define $f_v:\Pbf^1_v\to\RR$ by 
\begin{equation*}
\begin{split}
f_v(x) & =\log^+|x|_v-\hhat_{\phi_p,v}(x) \text{ when } x\in\Abf_v^1 \\
f_v(\infty) & =\lim_{x\to\infty}(\log^+|x|_v-\hhat_{\phi_p,v}(x)).
\end{split}
\end{equation*}
Note that $f_v$ is identically zero unless $v\in\{p,\infty\}$ by the remarks in Example~\ref{PhiPExample}.  We also note that $f_\infty(\infty)=\frac{\log p}{p-1}$ by Lemma~\ref{LocalHeightLem}.

Recall from Theorem~\ref{MainThm1} that for each $n\geq1$, $\alpha_n\in\Qbar$ satisfies $\phi_p^n(\alpha_n)=1$.  By a property of the canonical height, this imples that $\hhat_{\phi_p}(\alpha_n)=\frac{1}{p^n}\hhat_{\phi_p}(1)\to0$.  Since $f_v$ vanishes identically for $v\not\in\{p,\infty\}$, we have
\begin{equation*}
\begin{split}
h(\alpha_n)-\hhat_\phi(\alpha_n) & = \frac{1}{[\QQ(\alpha_n):\QQ]}\sum_{v\in M_\QQ}\sum_{\beta\in \Ocal_{\alpha_n}}f_v(\beta) \\
& = \sum_{v\in \{p,\infty\}}\frac{1}{[\QQ(\alpha_n):\QQ]}\sum_{\beta\in \Ocal_{\alpha_n}}f_v(\beta),
\end{split}
\end{equation*}
and since $\hhat_{\phi_p}(\alpha_n)\to0$, by the equidistribution theorem for small points we obtain
\begin{equation}\label{FinalLimForm}
\begin{split}
\lim_{n\to+\infty} h(\alpha_n) = \sum_{v\in \{p,\infty\}}\int f_vd\mu_{\phi_p,v}.
\end{split}
\end{equation}

In fact, we have
\begin{equation}\label{padicIntegral}
\begin{split}
\int f_pd\mu_{\phi_p,p}  & = 0
\end{split}
\end{equation}
because both $\log^+|x|_p$ and $\hhat_{\phi_p,p}(x) $ vanish on $F_{\phi_p,p}$ (which contains the support of $\mu_{\phi_p,p}$) by Lemma~\ref{LocalHeightLem} and the discussion in Example~\ref{PhiPExample}.  To calculate the Archimedean integral, we use the self-adjoint property of the Laplacian to obtain
\begin{equation}\label{ArchIntegral}
\begin{split}
\int f_\infty d\mu_{\phi_p,\infty}  & = -\int f_\infty d(\Delta \hhat_{\phi_p,\infty}) + f_\infty(\infty) \\
	& = -\int \hhat_{\phi_p,\infty} d(\Delta f_\infty) + f_\infty(\infty) \\
	& = \int \hhat_{\phi_p,\infty} d\mu_{\phi_p,\infty} -\int \hhat_{\phi_p,\infty} d\mu_{\sigma,\infty}  + f_\infty(\infty) \\
& = f_\infty(\infty) \\
& = \tfrac{\log p}{p-1}.
\end{split}
\end{equation}
Here we have used that $\int \hhat_{\phi_p,\infty} d\mu_{\phi_p,\infty} =0$, which follows from the fact (Lemma~\ref{LocalHeightLem}) that the local height $\hhat_{\phi_p,\infty}$ vanishes on the filled-Julia set of $\phi_p$ in $\CC$, which contains the support of $\mu_{\phi_p,\infty}$.  We have also used that $\int \hhat_{\phi_p,\infty} d\mu_{\sigma,\infty} =0$, which follows from the fact that $\hhat_{\phi_p,\infty}$ vanishes on the support of $\mu_{\sigma,\infty}$, which is the unit circle of $\CC$, as described in Example~\ref{SquaringEx}.  Indeed, if $|x|\leq1$ then $|\phi_p(x)|=\frac{1}{p}|x^p-x|\leq\frac{2}{p}\leq1$, and iterating shows that $\phi_p^n(x)$ is bounded as $n\to+\infty$.  In particular $\hhat_{\phi_p,\infty}(x)=0$ whenever $|x|=1$, and we conclude $\int \hhat_{\phi_p,\infty} d\mu_{\sigma,\infty}(x)=0$.

Combining (\ref{FinalLimForm}), (\ref{padicIntegral}), and (\ref{ArchIntegral}) concludes the proof that $h(\alpha_n)\to\frac{\log p}{p-1}$.
\end{proof}

\begin{rem}
To keep this paper more self-contained, we have chosen to prove Theorem~\ref{MainThm2} directly, without the use of  Theorem 1 of \cite{MR2869188}.  Alternatively, one can show that for a pair of polynomials $\phi(x),\psi(x)\in K[x]$ defined over a number field $K$, the Arakelov-Zhang pairing defined in \cite{MR2869188} can be expressed by either of the two expressions
\begin{equation}\label{AZPairing}
\langle\phi,\psi\rangle = \sum_{v\in M_K}r_v\int\hhat_{\phi,v}d\mu_{\psi,v} =\sum_{v\in M_K}r_v\int\hhat_{\psi,v}d\mu_{\phi,v},
\end{equation}
where $r_v=[K_v:\QQ_v]/[K:\QQ]$.  By Theorem 1 of \cite{MR2869188} we have $h(\alpha_n)\to\langle\sigma,\phi_p\rangle$, and in a calculation similar to the proof of Theorem~\ref{MainThm2} we have $\langle\sigma,\phi_p\rangle=\frac{\log p}{p-1}$.  The formula (\ref{AZPairing}) for the Arakelov-Zhang pairing in the polynomial case may be of interest as it is somewhat simpler than the more general treatment for rational maps described in \cite{MR2869188}.
\end{rem}



\def\cprime{$'$}

\end{document}